\documentclass{amsart}

\usepackage{pgf}
\usepackage{tikz}

\newtheorem{thm}{Theorem}[section]
\newtheorem{prop}{Proposition}[section]
\newtheorem{lem}{Lemma}[section]
\newtheorem{cor}{Corollary}[section]

\title[On periodic solutions in the Whitney's problem]{On 
periodic solutions in the Whitney's inverted pendulum problem}

\author{Roman Srzednicki}

\address{Institute of Mathematics, Faculty of Mathematics and Computer Science,
Jagiellonian University, ul. {\L}ojasiewicza 6, 30--348~Krak\'ow, Poland}
\email{srzednicki@im.uj.edu.pl}

\thanks{This research is partially supported by the Polish National Science Center
under Grant No. 2014/14/A/ST1/00453}



\date\today

\begin{document}

\begin{abstract}
In the book ``What is Mathematics?''
Richard Courant and Herbert Robbins presented
a solution of a Whitney's problem of an inverted
pendulum on a railway carriage moving on a straight line.
Since the appearance of the book in 1941 the solution
was contested by several distinguished mathematicians. 
The first formal proof based on
the idea of Courant and Robbins was published by
Ivan Polekhin in 2014. Polekhin also proved a theorem on
the existence of a periodic solution of the problem provided the
movement of the carriage on the line is periodic. 
In the present paper we slightly improve the Polekhin's
theorem by lowering the regularity class of the motion and
we prove 
a theorem on the existence of a periodic solution if 
the carriage moves periodically on the plane.
\end{abstract}

\maketitle

\section{Introduction}
\label{sec:intro}
In the year 1941,
in the first edition of
the book ``What is Mathematics'' Richard Courant and Herbert Robbins
posed the following question suggested by Hassler~Whitney:
``Suppose a train travels from station $A$ to station $B$ along a straight section
of track. The journey need not be of uniform speed or acceleration. The train
may act in any manner, speeding up, slowing down, coming to a halt, or even backing
up for a while, before reaching $B$. But the exact motion of the train is supposed
to be known in advance; that is, the function $s=f(t)$ is given, where $s$ is the distance
of the train from station $A$, and $t$ is the time, measured from the instant of departure.
On the floor of one of the cars a rod is pivoted so that it may move without friction
either forward or backward until it touches the floor. If it does touch the floor, we
assume that it remains on the floor henceforth; this will be the case if the rod does
not bounce. Is it possible to place the rod in such a position that, if it is released at
the instant when the train starts and allowed to move solely under the influence of
gravity and the potion of the train, it will not fall to the floor during the entire journey
form $A$ to $B$?''.  The question is illustrated in
Figure~\ref{fig:cart}. 
\begin{figure}[ht]
\begin{tikzpicture}
\draw[->] (-5,0) -- (5,0) ;
\fill (-4,0) circle (1pt) node [above]{$A$}; 
\fill (4,0) circle (1pt) node [above]{$B$}; 
\draw (-1.6,.4) circle [radius=.4];
\fill (-1.6,.4) circle (2pt);
\draw (1.6,.4) circle [radius=.4];
\fill (1.6,.4) circle (2pt);

\draw (-2.4,.5) -- (-2,.5);
\draw (-1.2,.5) -- (1.2,.5);
\draw (2,.5) -- (2.4,.5);

\draw (-2.4,.5) -- (-2.4,.9);
\draw (-2.4,.9) -- (-.1,.9);

\draw (2.4,.5) -- (2.4,.9);
\draw (2.4,.9) -- (.1,.9);
\draw (0,.9) circle (.1);

\draw [ultra thick] (-.8,2.3) -- (-1pt,1);
\end{tikzpicture}
\caption{ }
\label{fig:cart}
\end{figure}
Assuming continuous dependence of the motion
of the rod on its initial position, Courant and Robbins explained 
how the intermediate value theorem implies the positive answer.
Moreover, as exercises they posed the problems:  ``the reasoning above may be generalized 
to the case when the journey is of infinite duration" and
``generalize to the
case where the motion of the train is along any curve in the plane and the rod may fall
in any direction'' with a hint on nonexistence of a retraction of a disk onto its boundary
(compare \cite[pp.\,319 -- 321]{cr}).
\par
In the sequel we refer to the above question (either in finite or infinite
time setting) as to the 
Whitney's linear inverted pendulum problem (shorter: the problem or the Whitney's problem). 
If the curve $f$ lies in the plane, we call 
the problem ``planar''.
\par
In 1953 John E. Littlewood included the problem in his book ``A Mathematician's Miscellany'' 
(\cite[pp.\,12 -- 14]{lit}) and provided his own explanation
(close to the Courant and Robbins' one, actually). 
The fragment of \cite{cr} related to the problem 
was reproduced in \cite[pp.\,2412, 2413]{n4} 
(under the title ``The Lever of Mahomet'') in the year 1960. In the paper \cite{br} published in
1958 Arne~Broman noticed that
the assumption on continuity needs an explanation. He presented a comprehensive
argument supporting the Courant
and Robbins' solution, although his proof lacks of 
formal rigor at some details. 
\par
In 1976 the continuity assumption was contested 
by Tim Poston in the article  \cite{post} in Manifold, 
a mimeographed magazine issued by the Warwick University. 
Basing on possible (according to him) phase portraits related to the problem, he 
claimed that the rod can come arbitrarily close to the floor of the car and then swing back
causing discontinuity of its final position with respect to the initial one.
The argument of Poston was replicated by Ian Stewart in the book ``Game, Set and Math'' 
(\cite[pp.\,63, 64, 68]{st}, 1989)
and also in the Stewart's comments to the second edition of \cite{cr} which was published in 1996;
see \cite[pp.\,505 -- 507]{crs}. In the review of \cite{crs} 
published in The American Mathematical
Monthly in 1998 (see \cite{gil}) Leonard Gillman opposed to the arguments of Poston and Stewart writing:
``the acceleration of the train would have to be unbounded, which is not possible from physical
locomotive'' and 
presented a descriptive proof of the Courant and Robbins' solution suggested to him 
by Charles Radin. In 2001, in another review of \cite{crs} Leonard Blank also
criticized the Stewart's comment related to the problem by citing the conclusion
of the Gillman's report  (compare \cite{blank}). 
\par
The next comment contesting the continuity assumption
appeared in the Vla\-di\-mir Arnold's
short book ``What is Mathematics?'' published in 2002. Arnold writes (in my translation): 
``no continuous function -- the finite position for a given initial position -- 
can be seen immediately: it should be carefully defined (with the possibility of hitting the platform) 
and its continuity should be proved''. On the other hand, as the London Mathematical Society
Newsletter reported in 2009,
in the inaugural Christopher Zeeman Medal Award lecture 
entitled ``The Strange Case of the Courant-Robbins Train''
Stewart admitted that ``Courant and Robbins were
correct to assume continuity in the particular case where the carriage has a flat floor'' (compare \cite[pp.\,33, 34]{lms}). 
At that time Arnold had still objections towards the correctness of the solution
from \cite{cr}. In the
chapter ``Courant's Erroneous Theorems'' of the book \cite{a2}, after presentation of the problem
and the solution (for the travel time from $0$ to $T$ and the angle $\alpha\in [0,\pi]$ 
between the floor and the rod 
as a function of the initial position and time) he wrote:  ``many people disputed this (incorrect) proof, because 
even if a continuous function $\alpha(\cdot,T)$ of the initial position $\phi$ were defined, its
difference from $0$ to $\pi$ under the initial condition $\cdot=\phi$ would not imply that the
angle $\alpha$ differs from $0$ and $\pi$ at all intermediate moments of time $0<t<T$''.
\par
Finally, in 2014,
73 years from the announcement of the Whitney's problem,  in the paper
\cite{p1} Ivan Polekhin provided a short rigorous proof of the Courant and Robbins's solution
based on the Wa\.zewski retract theorem. Other proofs were
published in \cite{bk} and \cite{z}.
A natural question on the existence of a non-falling $T$-periodic solution when the path $f$
of the car is $T$-periodic was also considered by Polekhin. In \cite{p1}
he proved that if $f$ is of $C^3$-class then such a periodic solution
exist in the linear problem. Moreover, in \cite{p2} he got the same
conclusion in the planar problem, provided the rod
moves with friction.
\par
The main purpose of 
 the present paper is to prove two theorems on periodic solutions
 in the Whitney's problem. By lowering the regularity class of $f$, 
 the first one provides a minor improvement to the 
 corresponding result in \cite{p1}.
 
 \begin{thm}
 \label{thm:main_linear}
 If $f\colon \mathbb R\to \mathbb R$ is 
 $T$-periodic and of $C^2$-class then the linear
 Whitney's inverted pendulum problem has a $T$-periodic solution.
 \end{thm}
The second theorem is the main contribution of this research.
In contrast to \cite{p2}, it refers 
 to the original planar problem with the frictionless movement of the rod.
 \begin{thm}
 \label{thm:main_planar}
 If $f\colon \mathbb R\to \mathbb R^2$ is $T$-periodic 
 and of $C^3$-class then the planar Whitney's inverted pendulum
 problem has a $T$-periodic solution.
 \end{thm}
Assuming the mass of the rod is concentrated at its top (i.e. it is a mathematical
pendulum), both theorems are translated into 
theorems on periodic solutions of nonautonomous differential equations. In
the proofs we apply a result from the paper \cite{cmz} on the existence of periodic solutions
by a continuation method. 
Problems related to periodic perturbations of 
inverted pendulum are
considered also in control theory. In particular, in  \cite{ct1,ct2,ct3}
modifications of results of \cite{cmz}
were 
applied in the proofs of theorems on the exact tracking problem.
\par 
The rest of the paper is organized as follows. 
In Section~\ref{sec:some} we recall some standard notions corresponding to ordinary
differential equations: dynamical system, evolution operator, Poincar\'e operator, etc.
and also the notions of exit, entrance, and bound sets. The latter notion 
first appeared in a restricted context in \cite[p.\,42]{gm} and in full generality in \cite{zan}. 
Motivated by \cite[p.\,44]{gm} we introduce the notion of a curvature bound function which
is used to construct bound sets. The main result of this section is
Theorem~\ref{thm:fl} (which essentially is the same as \cite[Corollary~3]{cmz}),  
a sufficient condition for the existence of periodic solutions of a nonautonomous equation  
in terms of bound sets and the topological degree of a vectorfield homotopic to the right-hand
side of the equation. 
We provide a direct proof 
of Theorem~\ref{thm:fl}
based on the continuation 
of the fixed point index.
In Section~\ref{sec:plin} we 
derive a second-order equation related to the linear problem  by an elementary 
application of the Newton's second law in
the Cartesian coordinates, then we prove two lemmas related to the existence
of bound sets for suitable modifications of the derived equation, and finally
we apply Theorem~\ref{thm:fl} in a proof of Theorem~\ref{thm:lin}, a reformulation
of Theorem~\ref{thm:main_linear} in the Cartesian coordinates system. 
In exactly the same way we proceed in Section~\ref{sec:ps} on the planar problem; the main
result here is Theorem~\ref{thm:main} which reformulates Theorem~\ref{thm:main_planar}
in the Cartesian coordinates.
It should be noted, however, that in spite of similarities of the results
in Sections~\ref{sec:plin} and \ref{sec:ps}, some proofs in Section~\ref{sec:ps} are different
and essentially
more complex due to higher dimension of the phase space in the planar problem with
respect to the dimension of the  phase space in the linear one.
\par 
We use a standard vector notation in $\mathbb R^n$. In particular,
vectors are represented by columns, $A^T$ denotes the transpose
of $A$, $\operatorname{diag}(A_1,\ldots,A_k)$ denotes the block
diagonal matrix 
of square matrices $A_1,\ldots,A_k$, and
$\begin{bmatrix} x_1 &\ldots & x_k\end{bmatrix}$ denotes the matrix with columns
$x_1,\ldots,x_k$. The scalar product of vectors $x$ and $y$  is defined
as $x^Ty$ and 
the norm of a vector $x$ is given by $|x|:=\sqrt{x^Tx}$.
The derivative of a function $f$ is denoted by $D f$; if $f$ is single-variable it
is also denoted by $\dot f$. The Hessian of a scalar function $f$ is denoted by $D^2f$.
The norm of a continuous $T$-periodic function $f\colon \mathbb R\to 
\mathbb R^n$ is defined as $\|f\|:=\max_{t\in [0,T]}|f(t)|$. 
\par
The author wishes to thank an anonymous referee for pointing out an essential
error in the first version of the paper.

\section{A theorem on the existence of periodic solutions}
\label{sec:some}
Let $\Omega$ be an open set
in $\mathbb R^n$ and let $v\colon \Omega\to \mathbb R^n$ be a 
vectorfield of $C^1$-class. Denote by $\phi$ the dynamical
system generated by $v$; recall that 
$t\to \phi_t(x_0)$ is the maximal solution of the equation
\begin{equation}
\label{eq:v}
\dot x=v(x)
\end{equation}
with the initial value $x(0)=x_0$, 
$\phi_0(x)=x$, and $\phi_{s+t}(x)=\phi_s(\phi_t(x))$.
Let $E\subset \Omega$. The
\emph{entrance} and \emph{exit} sets of $E$ are given, respectively, a
\begin{align*}
&
E^+:=\{x\in E\colon \phi_{-\epsilon_n}(x)\notin E\
\text{for some}\ \{\epsilon_n\},\ 0<\epsilon_n\to 0\ 
\text{as $n\to\infty$}\},
\\
&E^-:=\{x\in E\colon \phi_{\epsilon_n}(x)\notin E\
\text{for some}\ \{\epsilon_n\},\ 0<\epsilon_n\to 0\ 
\text{as $n\to\infty$}\}.
\end{align*}
Clearly, $E^\pm$ are subsets $\partial E$, the boundary of  $E$.
A closed set $E$ in $\Omega$ 
is called a \emph{bound set} for 
the vectorfield $v$ if
for every $\epsilon>0$
there is no
$x\in\partial E$ such that $\phi_t(x)\in E$ for each 
$t\in (-\epsilon,\epsilon)$.
\begin{prop}
\label{prop:bounce}
If $E$ is closed and 
$\partial E=E^+\cup E^-$ then
$E$ is a bound set. 
\hfill\qed
\end{prop}
Let $U$ be an open subset of $\Omega$ and let
$e\colon U\to \mathbb R$ be of $C^2$-class. We call $e$ 
a \emph{curvature bound function} for $v$ if for
each $x\in U$ such that $e(x)=0$,
\begin{equation}
\label{eq:ext}
(De)v=0\ \Longrightarrow  
v^T(D^2e)v+(De)(Dv)v>0.
\end{equation}
Here (and also in the sequel) we use an abbreviate notation: we write
$v$ instead of $v(x)$, $e$ instead of $e(x)$, etc. whenever 
the choice of $x$ is clear
from the context.
\begin{prop}
\label{prop:etf}
If $e\colon U\to \mathbb R$ is a curvature bound function
then for each $x\in U$, $e(x)=0$ there exists an $\epsilon >0$
such that 
\begin{itemize}
\item[] either $e(\phi_t(x))>0$ for $t\in (-\epsilon,0)$, 
$e(\phi_t(x))<0$ for $t\in (0,\epsilon)$,
\item[] or $e(\phi_t(x))<0$ for $t\in (-\epsilon,0)$, 
$e(\phi_t(x))>0$ for $t\in (0,\epsilon)$,
\item[] or else $e(\phi_t(x))>0$ for $t\in (-\epsilon,0)\cup (0,\epsilon)$.
\end{itemize}
\end{prop}
\begin{proof}
The first two possibilities come from the inequalities $(De(x))v(x)<0$
and $(De(x))v(x)>0$. If $(De(x))v(x)=0$ then it follows by \eqref{eq:ext}
the function $t\to e(\phi_t(x))$ has a strict local minimum at $0$
(compare also \cite{handbook}, pp.\,617, 618).
\end{proof}
Let 
\[
E=\bigcap_{k=1,\ldots, r} \{e_k\leq 0\}
\]
for some continuous functions $e_k\colon \Omega\to \mathbb R$, hence $E$ is closed in $\Omega$.
As an immediate consequence of Propositions~\ref{prop:bounce} and \ref{prop:etf} we
get the following result.

\begin{cor}
\label{cor:waz}
If for some (possibly empty) 
closed subsets $Z_k\subset \partial E$, $k=1,\ldots,r$,
\begin{itemize}
\item[(a)] $e_k$ is a curvature bound function for $v$
in an open neighborhood
of the set 
\[
\{e_k=0\}\cap \partial E\setminus Z_k
\]
\item[(b)] $Z_k\subset E^-\cup E^+$,
\end{itemize}
then $E$ is a bound set for $v$.\hfill\qed
\end{cor}
Let $w\colon \mathbb R \times 
\Omega\to \mathbb R^n$ be of $C^1$-class. 
The vector-field 
$v\colon  
\mathbb R\times \Omega\to \mathbb R\times\mathbb R^n$ 
given by
\begin{equation}
\label{eq:vw}
v(t,x)=\begin{bmatrix} 1 \\ w(t,x)\end{bmatrix}.
\end{equation}
generates a dynamical system $\phi$ on $\mathbb R\times\Omega$
of the form
\[
\phi_t(t_0,x_0)=(t_0+t,\Psi(t_0,t_0+t,x_0)),
\]
where 
$t\to \Psi(t_0,t,x_0)$ is the maximal solution of the nonautonomous equation
\begin{equation}
\label{eq:xft}
\dot x=w(t,x)
\end{equation}
with the initial value
$x(t_0)=x_0$.
$\Psi$ is called the \emph{evolution operator} generated
by $w$. It satisfies $\Psi(t,t,x)=x$ and $\Psi(s,u,x)=\Psi(t,u,\Psi(s,t,x))$.
\par
Let $T>0$ and assume  
$t\to w(t,x)$ is $T$-periodic for every $x$, hence 
\[
\Psi(s+T,t+T,x)=\Psi(s,t,x).
\]
The map $P\colon x\to \Psi(0,T,x)$ is called the 
\emph{Poincar\'e operator}. Denote
by $\operatorname{Fix}(P)$ the set of all fixed points of $P$. 
\begin{prop}
\label{prop:poin}
$x_0\in \operatorname{Fix}(P)$ if and only if
$t\to \Psi(0,t,x_0)$ is a $T$-periodic solution of \eqref{eq:xft}.
\hfill\qed
\end{prop}
Let $E\subset \mathbb R\times \Omega$. 
For $t\in \mathbb R$ define $E_t:=\{x\in\mathbb R^n
\colon (t,x)\in E\}$. 
Assume $E$ is a bound set for $v$ and
$E_0$ is compact.  Set
\begin{align*}
&K:=\{x\in \operatorname{Fix}(P)\colon
\Psi(0,t,x)\in E_t\ \forall t\in [0,T]\},
\\
&L:=\{x\in \operatorname{Fix}(P)\cap E_0\colon
\exists t\in (0,T)\colon \Psi(0,t,x)\notin E_t\}.
\end{align*}
Clearly, $\operatorname{Fix}(P)\cap E_0=K\cup L$. 
Conditions imposed on $E$ imply that
both 
$K$ and $L$ are compact, $K\subset \operatorname{int} E_0$,
and
$K\cap L=\emptyset$. It follows that
there exists open sets $U,V$ in $E_0$, $U\subset \operatorname{int} E_0$, 
such that $U\cap V=\emptyset$ and 
\begin{align}
\label{eq:uk}
&U\cap \operatorname{Fix}(P)=K,
\\
\label{eq:vl}
&V\cap \operatorname{Fix}(P)=L.
\end{align} 
In particular, the \emph{fixed point index} of $P$ at $U$ is defined;
denote it by $\operatorname{ind}(P,U)$.
We refer  to \cite{dold} for its definition and properties. Actually, the excision property of the index imply
that $\operatorname{ind}(P,U)$ does not depend on the
choice of $U$ satisfying \eqref{eq:uk}.
\par
Let $\lambda\in [0,1]$ and $T>0$. Consider a continuous family of 
non-autonomous equations
\begin{equation}
\label{eq:xfl}
\dot x = w_\lambda(t,x),
\end{equation}
where $w_\lambda\colon \mathbb R \times 
\Omega\to \mathbb R^n$ is of $C^1$-class and 
$t\to w_\lambda(t,x)$ is $T$-periodic for every $x$. 
Through remainder of this section we adopt the above notation
concerning $w$ to $w_\lambda$ writing $v_\lambda$
(as in  \eqref{eq:vw}), $P_\lambda$,
$U_\lambda$,
etc.
\begin{prop}
\label{prop:2}
Let $E\subset \mathbb R\times \Omega$ and let $E_0$ be compact. If for every $\lambda\in [0,1]$, 
$E$ is a bound set 
for $v_\lambda$ then
\[
\operatorname{ind}(P_0,U_0)=\operatorname{ind}(P_1,U_1).
\]
\end{prop}
\begin{proof}
Fix $\lambda_0\in [0,1]$. 
There is an $\epsilon>0$ such that 
\[
|P_{\lambda_0}(x)-x|>\epsilon
\]
for $x\in E_0\setminus (U_{\lambda_0}\cup V_{\lambda_0})$.
Therefore, if $\lambda$
is sufficiently close to $\lambda_0$  then 
\[
\operatorname{Fix}(P_\lambda)\cap E_0\subset 
U_{\lambda_0}\cup V_{\lambda_0}
\]
and, by the homotopy property of the index,
\begin{equation}
\label{eq:ulv}
\operatorname{ind}(P_\lambda,U_{\lambda_0})=
\operatorname{ind}(P_{\lambda_0},U_{\lambda_0}).
\end{equation}
In order to finish the proof on should show that if $\lambda$
is close enough to $\lambda_0$ then
\begin{align}
\label{eq:klam} 
&K_\lambda \subset U_{\lambda_0},
\\ 
\label{eq:llam}
&L_\lambda\subset V_{\lambda_0}.
\end{align}
Indeed, in that case $K_\lambda\subset U_{\lambda_0}$, hence one can treat $U_{\lambda_0}$ as
$U_\lambda$ in the equation \eqref{eq:ulv} and 
therefore the function
$\lambda\to \operatorname{ind}(P_\lambda,U_\lambda)$
is locally constant for $\lambda\in [0,1]$, hence it is constant
and the result follows.
\par
For a proof of \eqref{eq:klam}
assume on the contrary that there exist $\lambda_n\to \lambda_0$
and $x_n\in V_{\lambda_0}\cap K_{\lambda_n}$. One can assume
$x_n\to x_0$. Since $\Psi_{\lambda_n}(0,t,x_n)\in E_t$ for
each $t\in [0,T]$ and $E$ is closed, 
$x_0\in \overline{V}_{\lambda_0}\cap K_{\lambda_0}$, which
is impossible. In a similar way the inclusion \eqref{eq:llam} follows.
\end{proof}
Assume now that $w$ is $t$-independent, i.e. $w\colon \Omega\to \mathbb R^n$
is a $C^1$-class vectorfield. In this case the Poincar\'e operator $P$ associated
with $T$-periodic solutions of \eqref{eq:xft} is equal to $\phi_T$, where
$\phi$
is the dynamical system generated by $w$.
Assume that
$E\subset \Omega$ is a bound set for $v$. Following the notation introduced
above we denote by 
$U$ be the open subset of $\operatorname{int} E$ satisfying
\eqref{eq:uk}.
By
$\operatorname{deg}(w,V,0)$ we denote the \emph{topological
degree} at $0$ of the $w$ in an open set $V$
(see \cite{deim} for the definition and properties). 

\begin{prop}
\label{prop:brouw}
$
\operatorname{ind}(\phi_T,U)=
(-1)^n\operatorname{deg}(v,\operatorname{int}E,0).
$
\end{prop}
\begin{proof} Let
$0<\epsilon\leq T$.
By assumptions and the argument in the proof of Proposition~\ref{prop:2},
\[
\operatorname{ind}(\phi_T,U)=
\operatorname{ind}(\phi_\epsilon,U_\epsilon)
\]
where $U_\epsilon\subset \operatorname{int} E$ is an open neighborhood of 
$\epsilon$-periodic points such that their orbits are contained in 
the interior of $E$.
If $\epsilon$ is small enough then
 \cite[Theorem~5.1]{fm} implies
\[
\operatorname{ind}(\phi_\epsilon,U_\epsilon)=
(-1)^n\operatorname{deg}(v,\operatorname{int}E,0),
\]
hence the result follows.
\end{proof}
Now we formulate the key theoretical result for the
proof of the existence of periodic solutions in the Whitney's problem.
As above, we consider the continuous family of equations 
\eqref{eq:xfl}, $\lambda\in [0,1]$, where $w_\lambda$ is 
$T$-periodic with respect to $t$.
\begin{thm}[compare Corollary~3 in \cite{cmz}]
\label{thm:fl}
Let $B$ be a compact subset of $\Omega$.
Assume $\mathbb R\times B$ is a bound
set for $v_\lambda$,
$0\leq \lambda\leq 1$. Assume moreover that $w_0$ is
$t$-independent
and 
\[
\operatorname{deg}(w_0,\operatorname{int}B,0)\neq 0.
\]
Then  the
equation \eqref{eq:xfl} for $\lambda=1$
has a $T$-periodic solution with image contained in $B$.
\end{thm}
\begin{proof}
This result is a direct consequence of 
Propositions~\ref{prop:poin}, \ref{prop:2}, and
 \ref{prop:brouw}. 
\end{proof}

\section{Periodic solutions in the linear Whitney's problem}
\label{sec:plin}
We begin with deriving the equation corresponding to the linear Whitney's problem
in the Cartesian coordinates. We assume that the whole mass $m$ of the rod is concentrated
at its top. 
Let $x(t)$ and $y(t)$ denote the horizontal and, respectively,
vertical position of the top of the rod with respect
to the pivot at time $t$ and let $\ell$ be equal to the length of the rod, i.e. to the distance from the
pivot to the top, thus
\begin{equation}
\label{eq:yt}
y(t)=\sqrt{\ell^2-x(t)^2}.
\end{equation}
The position of the pivot with respect to the
origin at time $t$ is equal to $f(t)$, hence the acceleration of the top
is equal to $\ddot x(t)+\ddot f(t)$. The constraint force is perpendicular to the arc
$y=\sqrt{\ell^2-x^2}$, hence the forces imposed on the top of the rod are given 
by the system of equations
\begin{align}
\label{eq:xx}
&m(\ddot x(t)+\ddot f(t))= \mu(t)x(t),
\\
\label{eq:yy}
&m \ddot y(t)=-mg +\mu(t)\sqrt{\ell^2-x(t)^2},
\end{align}
where $m>0$ is the mass of the rod, $g>0$ is the gravitational constant,
and $\mu$ is an unknown function such that $\ell\mu(t)$ is equal to 
the  magnitude of the constraint force at time $t$.
It follows by \eqref{eq:yt} and \eqref{eq:yy},
\begin{align*}
&\mu=m\frac{g+\ddot y}{\sqrt{\ell^2-x^2}},
\\
&\ddot y=- \frac{x\ddot x+\dot x^2}{\sqrt{\ell^2-x^2}}-\frac{x^2\dot x^2}{\left(\ell^2-x^2\right)^{3/2}},
\end{align*}
hence, by \eqref{eq:xx},
\[
\ddot x+\ddot f(t)=\frac{g}{\sqrt{\ell^2-x^2}}x-\frac{x^2\ddot x+x\dot x^2}{\ell^2-x^2}-
\frac{x^3\dot x^2}{\left(\ell^2-x^2\right)^2}
\]
which is equivalent to the nonautonomous equation
\begin{equation}
\label{eq:linear}
\ddot x=\left(\frac{g}{\ell^2}\sqrt{\ell^2-x^2}-\frac{\dot x^2}{\ell^2-x^2}\right)x
-\frac{\ell^2-x^2}{\ell^2}\ddot f(t).
\end{equation}
After changing the variable $x$ to $\frac{1}{\ell}x$ and rescaling $g$ and $\ddot f$ to
\[
G:=\frac{g}{\ell},\quad F:=\frac{\ddot f}{\ell},
\]
the equation \eqref{eq:linear} assumes a simpler form
\begin{equation}
\label{eq:sl}
\ddot x=\left(G\sqrt{1-x^2}-\frac{\dot x^2}{1-x^2}\right)x-(1-x^2)F(t).
\end{equation}
The equation \eqref{eq:sl} induces the system
\begin{align}
\label{eq:tlin}
&\dot t=1,
\\
\label{eq:xlin}
&\dot x=p,
\\
\label{eq:plin}
&\dot p=\left(G\sqrt{1-x^2}-\frac{p^2}{1-x^2}\right)x-\lambda(1-x^2)F(t)
\end{align}
with parameter $\lambda\in [0,1]$.
Following the notation used in Section~\ref{sec:some} we 
denote by 
$w_\lambda$ be the right-hand side of 
\eqref{eq:xlin},\eqref{eq:plin} and by 
$v_\lambda$
the right-hand side of \eqref{eq:tlin},\eqref{eq:xlin},\eqref{eq:plin}.
\par
We assume $F$ is continuous and $T$-periodic.
We are looking for a bound set for $v_\lambda$ of the form
$\mathbb R\times\Gamma_a\cap \Delta_b$, where $0<a<1$, $b>0$, and
\begin{align*}
&
\Gamma_a:=\{(x,p)\in\mathbb R\times 
\mathbb R\colon  |x|\leq a\},
\\
& \Delta_b=\{(x,p)\in \mathbb R\times 
\mathbb R\colon |x|<1,\ b|x|+|p|\leq b\}. 
\end{align*}

\begin{lem}
\label{lem:galin}
There exists an $a_0\in (0,1)$ such that if $a_0\leq a<1$ then
$\mathbb R\times \Gamma_a$ is a bound set for $v_\lambda$ for all $\lambda \in [0,1]$.  
\end{lem}
\begin{proof}
We apply 
Proposition~\ref{prop:bounce}. 
Let $t_0\in \mathbb R$ and $0<a<1$. 
If $p_0>0$ then $(t_0,a,p_0)
\in (\mathbb R\times\Gamma_a)^-$ since $\dot x(t_0)>0$ for the solution $t\to ((x(t),p(t))$
of the system \eqref{eq:xlin},\eqref{eq:plin} with the initial value 
$(x(t_0),p(t_0))=(a,p_0)$. Similarly, if $p_0<0$ then $(t_0,a,p_0)\in 
(\mathbb R\times \Gamma_a)^+$.  Now let $t\to (x(t),p(t))$ be the solution 
of \eqref{eq:xlin},\eqref{eq:plin} with the initial value
$(x(t_0),p(t_0))=(a,0)$. We assert that if
\begin{equation}
\label{eq:ggg}
Ga\sqrt{1-a^2}-\lambda(1-a^2)F(t_0)>0 
\end{equation}
then $(t_0,a,0)\in (\mathbb R\times \Gamma_a)^-\cap
(\mathbb R\times \Gamma_a)^+$.
Indeed, if \eqref{eq:ggg} holds then $\dot p(t_0)>0$, hence
there exists $\epsilon>0$ such that $p(t)<0$, hence $\dot x(t)<0$, for $t\in (t_0-\epsilon,t_0)$
and $p(t)>0$, hence $\dot x(t)>0$, for $t\in (t_0,t_0+\epsilon)$. 
This means $t\to x(t)$ has a strict local minimum at $t_0$ and the 
assertion follows.
In a similar way we treat the points of the form $(t_0,-a,p_0)$;
in the case $p_0=0$ the strict local minimum at $t_0$ is guaranteed if
\begin{equation}
\label{eq:ghh}
-Ga\sqrt{1-a^2}-\lambda(1-a^2)F(t_0)<0.
\end{equation}
It is clear that \eqref{eq:ggg} and \eqref{eq:ghh} are satisfied for all $\lambda\in [0,1]$ if 
\[
Ga-\|F\|\sqrt{1-a^2}>0
\]
which implies the conclusion.
\end{proof}

Through reminder of this section we assume $a$ satisfies
Lemma~\ref{lem:galin}.
\begin{lem}
\label{lem:2lin}
There exists $b>0$ such that the set $\mathbb R
\times\Gamma_a\cap \Delta_b$ 
is a bound set for $v_\lambda$ for all $\lambda\in [0,1]$.
\end{lem} 
\begin{proof}
Actually,
we prove that if 
\begin{equation}
\label{eq:wxp}
b^2> \frac{(1+a)\|F\|}{1-a}.
\end{equation}
then the conclusion holds. Indeed,
by Lemma~\ref{lem:2lin} it is enough to consider 
$(t,x,p)$ such that $b|x|+|p|=b$ and $|x|\leq a$.  
At first we assume $0\leq x\leq a$ and $p\geq 0$, hence $p=b(1-x)$.
We estimate the scalar product of $w_\lambda$ at $(t,x,a)$ with the vector 
$\begin{bmatrix} b & 1\end{bmatrix}^T$
perpendicular 
to the line $bx+p=b$ and directed outward from
$\Delta_b$. As a consequence of \eqref{eq:wxp} we get
\begin{multline*}
w_\lambda(t,x,p)^T\begin{bmatrix} b \\ 1 \end{bmatrix}
=bp+Gx\sqrt{1-x^2}-\frac{xp^2}{1-x^2}-\lambda(1-x^2)F(t)
\\
=b^2(1-x)\left(1-\frac{x}{1+x}-\lambda\frac{(1+x)F(t)}{b^2}\right)+
Gx\sqrt{1-x^2}
\\
>
b^2(1-a)\left(1-a-\frac{(1+a)\|F\|}{b^2}\right)>0,
\end{multline*}
hence 
$(t,x,p)\in (\mathbb R\times \Delta_b)^-$.
\begin{figure}[ht]
\begin{tikzpicture}
\draw[->] (-4,0) -- (4,0) node[below right] {$x$};
\draw (-3.7,-.1) -- (-3.7,.1);
\draw (-3.7,-.3) node {$-1$};
\draw (3.7,-.1) -- (3.7,.1);
\draw (3.7,-.3) node {$1$};
\draw[->] (0,-4) -- (0,4) node[right] {$p$};
\draw (-2.4,-1.5) -- (-2.4,1.5);
\draw (-2.1,-.3) node {$-a$};
\draw (2.4,-1.5) -- (2.4,1.5);
\draw (2.2,-.3) node {$a$};
\draw (-2.4,1.5) -- (0,3.4);
\draw (-2.4,-1.5) -- (0,-3.4);
\draw (2.4,1.5) -- (0,3.4);
\draw (2.4,-1.5) -- (0,-3.4);
\draw (-.2,3.7) node {$b$};
\draw (.3,-3.8) node {$-b$};

\draw[domain=-3:-2.4,smooth]  plot(\x,{sqrt(-\x-2.4) });
\draw[domain=-3:-2.4,smooth,<-]  plot(\x,{-sqrt(-\x-2.4) });

\draw[domain=2.4:3,smooth,->]  plot(\x,{sqrt(\x-2.4) });
\draw[domain=2.4:3,smooth,]  plot(\x,{-sqrt(\x-2.4) });

\draw[domain=-2.8:-2,smooth,->]  plot(\x,{sqrt(-\x-1) });
\draw[domain=-2.8:-2,smooth,<-]  plot(\x,{-sqrt(-\x-1) });

\draw[domain=2:2.8,smooth,->]  plot(\x,{sqrt(\x-1) });
\draw[domain=2:2.8,smooth,<-]  plot(\x,{-sqrt(\x-1) });

\draw[domain=-1.6:-.8,smooth,->]  plot(\x,{-.3*\x+2.1 });\draw[domain=-1.6:-.8,smooth,<-]  plot(\x,{.3*\x-2.1 });

\draw[domain=.8:1.6,smooth,->]  plot(\x,{.3*\x+2.1 });
\draw[domain=.8:1.6,smooth,<-]  plot(\x,{-.3*\x-2.1 });

\draw[domain=-.8:.8,smooth,->]  plot(\x,{.5*\x*\x+3.4 });
\draw[domain=-.8:.8,smooth,<-]  plot(\x,{-.5*\x*\x-3.4 });
\end{tikzpicture}
\caption{ }
\label{fig:1}
\end{figure}
The same inequality \eqref{eq:wxp} implies 
$(t,x,p)\in (\mathbb R\times \Delta_b)^-$ if $-a\leq x\leq 0$ and
$p\leq 0$, and $(t,x,p)\in \mathbb (\mathbb R\times \Delta_b)^+$
if $0\leq x\leq a$ and
$p\leq 0$ or $-a\leq x\leq 0$ and $p\geq 0$, as it is sketched out in
Figure~\ref{fig:1}, hence the result follows by Proposition~\ref{prop:bounce}.
\end{proof}

\begin{thm}
\label{thm:lin}
If $F$ is continuous and $T$-periodic then \eqref{eq:sl} has a $T$-periodic solution.
\end{thm}
\begin{proof} 
The equation \eqref{eq:sl} is equivalent to the 
system \eqref{eq:xlin},\eqref{eq:plin} for $\lambda=1$. 
Assume first $F$ is of $C^1$-class; we prove that there exists a $T$-periodic solution
with image contained in $\Gamma_a\cap \Delta_b$.
By Theorem~\ref{thm:fl} and Lemma~\ref{lem:2lin}
it remains to prove
that the degree of $w_0$ 
in the interior of $\Gamma_a\cap \Delta_b$ is not equal to zero.  The origin $(0,0)$ is the only zero of $w_0$ and
\[
(Dw_0)(0,0)=\begin{bmatrix}0 & 1\\ G & 0\end{bmatrix},
\]
hence
\[
\operatorname{deg}(w_0,\operatorname{int}
(\Gamma_a\cap \Delta_b),0)=\operatorname{sign} \det (Dw_0)(0,0)=-1
\]
and the proof is complete in the $C^1$-class case. Since $\Gamma_a\cap \Delta_b$ is compact,
a standard approximation argument provides a proof if $F$ is continuous.
\end{proof}  

Actually, in the above proof we do not need to assume first that $F$ is
of $C^1$-class. That assumption was required since Theorem~\ref{thm:fl} was formulated
in the $C^1$-class setting applied for the whole Section~\ref{sec:some}. 
A general formulation of that result
given in \cite{cmz} is valid even for equations satisfying the Carath\'eodory conditions. 
\par
The phase portrait shown in Figure~\ref{fig:1} indicates an alternative
proof of Theorem~\ref{thm:lin} by an application of \cite[Theorem~1]{uiam} 
(see also \cite[Corollary~7.4]{rs-na}).  As it was
mentioned in Section~\ref{sec:intro}, Theorem~\ref{thm:lin} is the reformulation of
Theorem~\ref{thm:main_linear} in the Cartesian coordinates system, hence
the latter theorem is also proved.

\section{Periodic solutions in the planar Whitney's problem}
\label{sec:ps}
Now we consider the planar Whitney's problem. We proceed in an analogous 
way as in Section~\ref{sec:plin}. At first we derive the corresponding differential equations.
The top of the rod at time $t$ has the horizontal coordinates 
$x(t)=(x_1(t),x_2(t))$ with respect to the pivot, its distance from the pivot
is equal to $\ell$, and the position of the pivot
is represented by $f(t)=(f_1(t),f_2(t))$. 
The vertical position is equal to $y(t)$,
hence
\begin{equation}
\label{eq:yplan}
y(t)=\sqrt{\ell^2-|x(t)|^2}.
\end{equation}
The constraint force is perpendicular to the half-sphere 
$y=\sqrt{\ell^2-|x|^2}$. Denote by $m$ the mass of the rod; we assume
it is concentrated at the top. Let 
$g$ be the gravitational constant and let $\mu(t)$ refers to the magnitude of
the constraint force as in \eqref{eq:xx},\eqref{eq:yy}, hence
\begin{align}
\label{eq:mxpl}
& m(\ddot x(t)+\ddot f(t))=\mu(t)x(t),
\\
\label{eq:mypl}
& m\ddot y(t)=-mg+\mu(t)\sqrt{\ell^2-|x(t)|^2}.
\end{align}
The equations \eqref{eq:yplan} and \eqref{eq:mypl} imply
\begin{align*}
&
\mu=m\frac{g+\ddot y}{\sqrt{\ell^2-|x|^2}},
\\
&
\ddot y=
-\frac{x^T\ddot x  + |\dot x|^2}{\sqrt{\ell^2-|x|^2}}
-\frac{(x^T\dot x)^2}{(\ell^2-|x|^2)^{3/2}},
\end{align*}
hence, by \eqref{eq:mxpl}, the system 
\begin{align*}
%%\label{eq:fx1}
&(\ell^2-x_2^2)\ddot x_1+x_1x_2\ddot x_2=
\left(
g\sqrt{\ell^2-|x|^2}- \frac{(x^T\dot x)^2}{\ell^2-|x|^2}-
|\dot x|^2\right)x_1- (\ell^2-|x|^2)\ddot f_1,
\\
%%\label{eq:fx2}
&x_1x_2\ddot x_1+(\ell^2-x_1^2)\ddot x_2=
\left(
g\sqrt{\ell^2-|x|^2}- \frac{(x^T\dot x)^2}{\ell^2-|x|^2}-
|\dot x|^2\right)x_2- (\ell^2-|x|^2)\ddot f_2
\end{align*}
which resolves to the equation
\begin{equation}
\label{eq:plan}
\ddot x=\frac{1}{\ell^2}
\left(g\sqrt{\ell^2-|x|^2}-\frac{(x^T\dot x)^2}{\ell^2-|x|^2}-|\dot x|^2\right)x 
+\frac{x^T\ddot f(t)}{\ell^2}x-\ddot f(t).
%%%- \frac{\ell^2-|x|^2}{\ell^2}\ddot f(t).
\end{equation}
After changing the variable $x$ to $\frac{1}{\ell}x$, 
the equation \eqref{eq:plan}
becomes
\begin{equation}
\label{eq:main}
\ddot x=
\left(G\sqrt{1-|x|^2}-\frac{(x^T\dot x)^2}{1-|x|^2}-|\dot x|^2\right)x 
+(x^TF(t))x-F(t),
%%- (1-|x|^2)F(t),
\end{equation}
where $G:=\frac{g}{\ell}$ and $F:=\frac{1}{\ell}\ddot f$.
We associate with \eqref{eq:main}
the autonomous system 
\begin{align}
\label{eq:t}
&\dot t=1,
\\
\label{eq:x}
&\dot x=p,
\\
\label{eq:p}
&\dot p=\left(G\sqrt{1-|x|^2}-\frac{(x^Tp)^2}{1-|x|^2}-|p|^2\right)x 
+\lambda((x^TF(t))x-F(t))
%%%- \lambda(1-|x|^2)F(t)
\end{align}
with parameter $\lambda\in [0,1]$.
Set
\begin{align*}
&R:=G\sqrt{1-|x|^2}-\frac{(x^Tp)^2}{1-|x|^2}-|p|^2,
\\
&\Phi:=
\lambda((x^TF(t))x-F(t)).
%%%\lambda(1-|x|^2)F(t). 
\end{align*}
Assume $F$ is $T$-periodic and of $C^1$-class.
As in Section~\ref{sec:plin},
we 
denote by $w_\lambda$ be the right-hand side of \eqref{eq:x},\eqref{eq:p} and by $v_\lambda$
the right-hand side of \eqref{eq:t},\eqref{eq:x},\eqref{eq:p};
\[
w_\lambda(t,x,p):=\begin{bmatrix} p \\ Rx+\Phi \end{bmatrix},
\quad
v_\lambda(t,x,p):=\begin{bmatrix} 1\\ p \\ Rx+\Phi \end{bmatrix}.
\]
The derivative of $v_\lambda$ is given by
\[
Dv_\lambda=\begin{bmatrix} 0 & 0 & 0 
\\
0 & 0 & I
\\
\frac{\partial\Phi}{\partial t} & RI+x\frac{\partial R}{\partial x} +
\frac{\partial \Phi}{\partial x}
& x\frac{\partial R}{\partial p}
\end{bmatrix},
\]
hence
\begin{equation}
\label{eq:dvv}
(Dv_\lambda)v_\lambda=\begin{bmatrix} 0 \\ Rx+\Phi \\ 
\frac{\partial\Phi}{\partial t}+Rp+x\frac{\partial R}{\partial x}p +
\frac{\partial \Phi}{\partial x}p 
+Rx\frac{\partial R}{\partial p}x+ x\frac{\partial R}{\partial p}\Phi
\end{bmatrix}
\end{equation}
Similarly as it was done in Section~\ref{sec:plin} in the 
$2$-dimensional setting,
for $0<a<1$ and $b>0$ 
we define subsets of $\mathbb R^4$,
\begin{align*}
&
\Gamma_a:=\{(x,p)\in\mathbb R^2\times 
\mathbb R^2\colon  |x|\leq a\},
\\
& \Delta_b=\{(x,p)\in \mathbb R^2\times 
\mathbb R^2\colon |x|<1,\ b|x|+|p|\leq b\}.
\end{align*}
Let
$m_a(t,x,p):=\frac{1}{2}|x|^2-\frac{1}{2}a^2$, hence
\[
\{m_a\leq 0\}=\mathbb R\times \Gamma_a.
\]
\begin{lem}
\label{lem:ga}
There exists an $a_0\in (0,1)$ such that if $a_0\leq a<1$ then
$m_a$ is a curvature bound function for $v_\lambda$ for all $\lambda \in [0,1]$.  
\end{lem}
\begin{proof}
Let $m_a(t,x,p)=0$, i.e. $|x|=a$. Assume 
\begin{equation}
\label{eq:dmv}
(Dm_a)v_\lambda=0
\end{equation}
at $(t,x,p)$. Since
\[
D m_a=\begin{bmatrix} 0 & x^T & 0\end{bmatrix},\quad D^2m_a=\operatorname{diag}(0, I, 0),
\]
the equations \eqref{eq:dvv} and \eqref{eq:dmv} imply
$x^Tp=0$ and  
\begin{multline*}
v_\lambda^T(D^2m_a)v_\lambda+(Dm_a)(Dv_\lambda)v_\lambda
\\
=|p|^2+R|x|^2+x^T\Phi\geq
|p|^2+Ga^2\sqrt{1-a^2}-|p|^2a^2-\lambda a(1-a^2)|F(t)|
\\
\geq (1-a^2)|p|^2+
a\sqrt{1-a}\left(Ga\sqrt{1+a}-\lambda(1+a)\|F\|\sqrt{1-a}\right).
\end{multline*}
Clearly, there exists $a_0$ close to $1$ 
such that if $a$ satisfies $a_0\leq a<1$ then 
the right-hand side is positive for all $\lambda\in [0,1]$, hence the result follows. 
\end{proof}
In the sequel we fix an $a$ satisfying the conclusion of Lemma~\ref{lem:ga}.

\begin{lem}
\label{lem:2}
There exists $b>0$ such that the set $\mathbb R
\times\Gamma_a\cap \Delta_b$ 
is a bound set for $v_\lambda$ for all $\lambda\in [0,1]$.
\end{lem} 
\begin{proof} 
Let $n_b(t,x,p):=b|x|+|p|-b$, hence
\[
\{n_b\leq  0\}=\mathbb R\times \Delta_b.
\] 
It follows the set $\mathbb R
\times\Gamma_a\cap \Delta_b$ is of the form required in Corollary~\ref{cor:waz}
for $e_1=m_a$ and $e_2=n_b$. By Lemma~\ref{lem:ga} and the choice of $a$, in order
to apply the theorem it is enough to examine the part of the boundary corresponding
to $n_b=0$. 
Set $Z:=\{(t,0,p)\colon t\in\mathbb R,\ |p|=b\}$. Clearly, $n_b$ is of $C^2$-class
in a neighborhood of each point 
$(t,x,p)\in (\mathbb R\times \partial \Delta_b)\setminus Z$.
At first we
find an estimate on $b$ which guarantee 
that $Z$ is contained
in the exit set of $\mathbb R\times\Delta_b$. Let $(t,0,p)\in Z$. It suffices
to find $b$ such that $v_\lambda(t,0,p)\notin C$, where $C$ denotes the cone
at $(t,0,p)$ of vectors directed to $\mathbb R\times \Delta_b$,
which means
\[
\begin{bmatrix}
1 \\ p \\ -\lambda F(t)\end{bmatrix}
\neq \mu \begin{bmatrix} u \\ x \\ -p\end{bmatrix}
\]
for each $\mu>0$, $u\in\mathbb R$, and $|x|\leq 1$. This is
satisfied if 
\begin{equation}
\label{eq:estz}
b^2>\|F\|.
\end{equation}
It remains to find values of $b$ for which the implication \eqref{eq:ext}
with $e=n_b$ and $v=v_\lambda$ is satisfied at each point  
$
(t,x,p)\in (\mathbb R\times \Gamma_a\cap \partial \Delta_a)\setminus Z,
$
i.e. at each $(t,x,p)$ such that
$t\in\mathbb R$, $0<|x|\leq a$, and
$|p|=b(1-|x|)$. Direct calculations show
\[
Dn_b=\begin{bmatrix} 0 & \frac{b}{|x|}x^T & \frac{1}{|p|}p^T\end{bmatrix},
\quad
D^2n_b=\operatorname{diag}(0,A,B),
\]
where
\[
A:=\frac{b}{|x|^3}\begin{bmatrix}x_2^2 & -x_1x_2 \\ -x_1x_2 & x_1^2\end{bmatrix},\quad
B:=\frac{1}{|p|^3}\begin{bmatrix}p_2^2 & -p_1p_2 \\ -p_1p_2 & p_1^2\end{bmatrix},
\]
hence
\begin{multline}
\label{eq:dnv}
(Dn_b)v_\lambda
\\
=\left(\frac{b}{|x|}+\frac{G}{|p|}\sqrt{1-|x|^2}-\frac{(x^Tp)^2}{|p|(1-|x|^2)}-|p|\right)x^Tp
+\lambda\frac{x^TF(t)x^Tp-F(t)^Tp}{|p|}
%%%-\frac{\lambda(1-|x|^2)}{|p|}F(t)^Tp
\end{multline}
and, by \eqref{eq:dvv},
\begin{multline}
\label{eq:vdv}
v_\lambda^T(D^2n_b)v_\lambda+(Dn_b)(Dv_\lambda)v_\lambda
\\
=
\frac{b}{|x|^3}\left(\det\begin{bmatrix} x & p \end{bmatrix}\right)^2
+ \frac{1}{|p|^3}\left(\det\begin{bmatrix} p & Rx+\Phi \end{bmatrix}\right)^2
\\
+bR|x|+R|p|+
\frac{1}{|p|}\left(\frac{\partial R}{\partial x}p+R\frac{\partial R}{\partial p}x\right)x^Tp
\\
+\frac{b}{|x|}x^T\Phi + \frac{1}{|p|}p^T
\left(\frac{\partial\Phi}{\partial t}+\frac{\partial \Phi}{\partial x}p\right)
+\frac{x^Tp}{|p|}\frac{\partial R}{\partial p}\Phi.
\end{multline}
We are looking for  $b$ 
such that  the value of the right-hand side of \eqref{eq:vdv} is positive provided $(Dn_b)v_\lambda=0$. 
Since
\begin{multline*}
\frac{b}{|x|}-\frac{(x^Tp)^2}{|p|(1-|x|^2)}-|p|
\geq
\frac{b}{|x|}-\frac{|x|^2|p|^2}{|p|(1-|x|^2)}-|p|
\\
=\frac{b}{|x|}-\frac{|p|}{1-|x|^2}=
\frac{b}{|x|}-\frac{b}{1+|x|}=\frac{b}{|x|(1+|x|)}>0,
\end{multline*}
by \eqref{eq:dnv} the equation $(Dn_b)v_\lambda=0$ implies
\begin{equation}
\label{eq:xtp}
|x^Tp|\leq \frac{|x|(1+|x|)(1+|x|^2)\|F\|}{b}\leq \frac{4\|F\||x|}{b}.
%%%%|x^Tp|\leq \frac{|x|(1+|x|)(1-|x|^2)\|F\|}{b}\leq \frac{2\|F\||x|}{b}.
\end{equation}
Assume $b$ satisfies
\begin{equation}
\label{eq:bff}
b^4> \frac{16\|F\|^2}{(1-a)^3}.
%%b^4> \frac{4\|F\|^2}{(1-a)^3}.
\end{equation}
It follows, in particular, that 
\eqref{eq:estz} also holds.
As a consequence of \eqref{eq:xtp} and \eqref{eq:bff} we get
\begin{multline}
\label{eq:d+brx+rp}
\frac{b}{|x|^3}\left(\det\begin{bmatrix} x & p \end{bmatrix}\right)^2
+bR|x|+R|p|=
b\frac{|p|^2}{|x|}-\frac{b}{|x|^3}(x^Tp)^2+bR
\\
> 
\frac{b}{|x|}\left((1-|x|)|p|^2 -\frac{16\|F\|^2}{b^2}\right)-\frac{16\|F\|^2|x|^2}{(1-|x|^2)b}
\\
\geq 
\frac{b}{|x|}\left((1-a^3)b^2 -\frac{16\|F\|^2}{b^2}\right)-\frac{16\|F\|^2a^2}{(1-a^2)b}
\geq
Kb^3-L\frac{1}{b},
\end{multline} 
where
\[
K:=\frac{(1-a)^3}{a},\quad L:=16\|F\|^2\left( \frac{1}{a}+\frac{a^2}{1-a^2}\right).
\]
Now we examine the remaining terms of the right-hand side of \eqref{eq:vdv}. 
The term $\frac{1}{|p|^3}\left(\det\begin{bmatrix} p & Rx+\Phi \end{bmatrix}\right)^2$ is always
non-negative.
Since
\begin{align*}
&\frac{\partial R}{\partial x}=-\frac{G}{\sqrt{1-|x|^2}}x^T-\frac{2x^Tp}{1-|x|^2}p^T-
\frac{2(x^Tp)^2}{(1-|x|^2)^2}x^T,
\\
&\frac{\partial R}{\partial p}=-\frac{2x^Tp}{1-|x|^2}x^T-2p^T,
\end{align*}
as an application of \eqref{eq:xtp} we get
\begin{multline}
\label{eq:prpx}
\left|\frac{1}{|p|}\left(\frac{\partial R}{\partial x}p+R\frac{\partial R}{\partial p}x\right)x^Tp\right|
\\
\leq \frac{G}{|p|\sqrt{1-|x|^2}}(x^Tp)^2
+ \frac{2|p|}{1-|x|^2}(x^Tp)^2
+ \frac{2}{|p|(1-|x|^2)^2}|x^Tp|^3
\\
+ \frac{2G|x|^2}{|p|\sqrt{1-|x|^2}}(x^Tp)^2
+ \frac{2|x|^2}{|p|(1-|x|^2)^2}(x^Tp)^4
+ \frac{2|p||x|^2}{1-|x|^2}(x^Tp)^2
\\
+\frac{2G\sqrt{1-|x|^2}}{|p|}(x^Tp)^2
+\frac{2}{|p|(1-|x|^2)}(x^Tp)^4
+2|p|(x^Tp)^2
\\
\leq M\left(\frac{1}{b}+\frac{1}{b^3}+\frac{1}{b^4}+\frac{1}{b^5}\right)
\end{multline}
for some non-negative constant $M$ depending on $a$, $\|F\|$, and $G$. Furthermore,
\begin{align}
\label{eq:bf}
&\left|\frac{b}{|x|}x^T\Phi\right|\leq \|F\|b,
\\
\label{eq:f'}
& \left|\frac{1}{|p|}p^T\frac{\partial\Phi}{\partial t}\right|\leq 2\|\dot F\|.
\end{align}
Since
%%%$\frac{\partial \Phi}{\partial x}=-2\lambda F(t)x^T$, 
\[
\frac{\partial \Phi}{\partial x}=\lambda((x^TF(t))I+xF^T),
\]
one gets
\begin{equation}
\label{eq:pfxp}
\left|\frac{1}{|p|}p^T\frac{\partial \Phi}{\partial x}p\right|
\leq 
2\|F\||x||p| < 2\|F\|b.
\end{equation}
%%the inequality \eqref{eq:xtp} implies
%%\begin{equation}
%%\label{eq:pfxp}
%%\left|\frac{1}{|p|}p^T\frac{\partial \Phi}{\partial x}p\right|
%%\leq 
%%2\|F\||x^Tp|\leq 4a\|F\|^2\frac{1}{b}.
%%\end{equation}
Finally, \eqref{eq:xtp} implies
\begin{multline}
\label{eq:xtprf}
\left| \frac{x^Tp}{|p|}\frac{\partial R}{\partial p}\Phi\right|
\\
\leq 2\frac{|x^Tp|}{|p|}
\left(\frac{|x^Tp||x^TF(t)||x|^2}{1-|x|^2}
+ \frac{|x^Tp||x^TF(t)|}{1-|x|^2}+|x^Tp||x^TF(t)|+|p^TF(t)|\right)
\\
\leq N\left(\frac{1}{b}+\frac{1}{b^2}\right)
%%\frac{\lambda}{|p|}\left|2(x^Tp)^2x^TF(t)+2(1-|x|^2)(x^Tp)p^TF(t)\right| 
%%\\
%%\leq \frac{4|x|^2\|F\|^2}{|p|b}+\frac{4(1-|x|^2)|x|\|F\|^2}{b}\leq N\left(\frac{1}{b}+\frac{1}{b^2}\right)
\end{multline}
for some non-negative constant $N$ depending on $a$ and $\|F\|$. 
Since the constant $K$ is positive,
the equation \eqref{eq:vdv} together with the estimates \eqref{eq:d+brx+rp}
-- \eqref{eq:xtprf} imply
\begin{multline*}
v_\lambda^T(D^2n_b)v_\lambda+(Dn_b)(Dv_\lambda)v_\lambda
>
Kb^3
-3\|F\|b - 2\|\dot F\|-\left(L+M+N\right)\frac{1}{b}
\\
-N\frac{1}{b^2}-
M\left(\frac{1}{b^3}+\frac{1}{b^4}+\frac{1}{b^5}\right)
>0
\end{multline*}
provided $b$ satisfying \eqref{eq:bff} (hence also satisfying \eqref{eq:estz})
is large enough. Now the conclusion is an immediate consequence of 
Corollary~\ref{cor:waz}. 
\end{proof}

\begin{thm}
\label{thm:main}
If $F$ is $T$-periodic and of $C^1$-class then \eqref{eq:main} has a $T$-periodic solution.
\end{thm}
\begin{proof} 
We apply the same argument as in the proof of Theorem~\ref{thm:lin}.
The equation \eqref{eq:main} is equivalent to the 
system \eqref{eq:x},\eqref{eq:p} for $\lambda=1$. Since
\[
(Dw_0)(0,0)=\begin{bmatrix}0 & 0 & 1 & 0\\ 0 & 0 & 0 &1 \\ G & 0 & 0 & 0 \\ 0 & G & 0 & 0\end{bmatrix}
\]
and $(0,0)$ is the only zero of $w_0$, 
\[
\operatorname{deg}(w_0,\operatorname{int}
(\Gamma_a\cap \Delta_b),0)=\operatorname{sign} \det (Dw_0)(0,0)=1,
\]
hence the result is a consequence of
Theorem~\ref{thm:fl} and Lemma~\ref{lem:2}.
\end{proof}  
By proving Theorem~\ref{thm:main} we simultaneously provided a proof of Theorem~\ref{thm:main_planar},
the main result of the present paper.

\end{document}